\documentclass[12pt]{amsart}
\usepackage[english]{babel}
\usepackage{amsfonts,amssymb,latexsym,amscd}

\theoremstyle{plain}
\newtheorem{theorem}{Theorem}
\newtheorem{lemma}{Lemma}

\theoremstyle{definition}
\newtheorem{definition}{Definition}
\theoremstyle{remark}

\usepackage{hyperref}

\begin{document}

\title{On equivariant fibrations of G-CW-complexes}
\author{Pavel S. Gevorgyan}
\address{Moscow Pedagogical State University}
\email{pgev@yandex.ru}

\author{Rolando Jimenez}
\address{Institute of Mathematics, National Autonomous University of Mexico, Oaxaca, Mexico}
\email{rolando@matcuer.unam.mx}

\begin{abstract}
It is proved that if $G$ is a compact Lie group, then an equivariant Serre fibration of $G$-CW-complexes is an equivariant Hurewicz fibration in the class of compactly generated $G$-spaces. In the nonequivariant setting, this result is due to Steinberger, West and Cauty. The main theorem is proved using the following key result: a $G$-CW-complex can be embedded as an equivariant retract in a simplicial $G$-complex. It is also proved that an equivariant map $p: E \to B$ of $G$-CW-complexes is a Hurewicz $G$-fibration if and only if the $H$-fixed point map $p^H : E^H \to B^H$ is a Hurewicz fibration for any closed subgroup $H$ of $G$. This gives a solution to the problem of James and Segal in the case of $G$-CW-complexes.
\end{abstract}



\maketitle

\footnotetext[1]{The research of R. Jimenez was partially supported by the Programa de Apoyos para la Superacion del Personal Academico de la Direccion General de Asuntos del Personal Academico de la Universidad Nacional Autonoma de Mexico, the Secretaria de Educacion Publica and the Consejo Nacional en Ciencia y Tecnologia (grant no. 284621).}

\section{Introduction}

Steinberger and West proved in \cite{St-West} that a Serre fibration of CW-complexes is a Hurewicz fibration in the class of compactly generated spaces. This is not true in the class of all topological spaces. Cauty \cite{Cauty-2} found and corrected a mistake in the proof of Steinberger and West. May and Sigurdsson conjectured in \cite{May-Sig} that the Steinberger-West-Cauty theorem is valid in the equivariant setting for compact Lie group actions. This conjecture is solved in the positive in the present paper, by showing that an equivariant Serre fibration of $G$-CW-complexes is an equivariant Hurewicz fibration in the class of compactly generated G-spaces (Theorem \ref{th-rassloenie}). To prove this theorem, an important result is established on the embeddability of a $G$-CW-complex as an equivariant retract in a simplicial $G$-complex (Theorem \ref{th-retract}).

It is known that if  $p:E\to B$ is an equivariant fibration, then the induced map of $H$-fixed points $p^H : E^H \to B^H$ is a fibration in the usual sense for all closed subgroups $H$ of $G$. James and Segal asked in \cite{james-segal} whether this necessary condition is also sufficient. In the case of finite group actions, such a result was obtained by Bredon (see \cite{Bredon}, Ch. III, 4.1). Using the main result in this paper, Theorem \ref{th-rassloenie}, we give a positive solution to the problem of James and Segal \cite{james-segal} in the case when $E$ and $B$ are $G$-CW-complexes. Namely, we prove that an equivariant map$p:E\to B$ is a Hurewicz $G$-fibration if and only if $p^H : E^H \to B^H$ is a Hurewicz fibration for any closed subgroup $H$. In the general case, when $E$ and $B$ are not $G$-CW-complexes, this statement is most likely false.

In this paper we use some results and ideas due to Cauty \cite{Cauty-2} and Illman \cite{Illman-1}.

In what follows we assume that $G$ is a compact Lie group, although the reader will easily notice that some results are true under the weaker assumption that $G$ is a compact topological group.

\section{Embeddings of $G$-CW-complexes in simplicial $G$-complexes}

Simplicial $G$-complexes and $G$-CW-complexes are important objects of equivariant algebraic topology. These spaces are constructed using simpler objects: equivariant simplices and equivariant cells. We recall the construction of these objects following the work \cite{Illman-1} of Illman.

Let $H_0$, $H_1$, \ldots \,, $H_n$ be closed subgroups of $G$ such that $H_0\supset H_1\supset \ldots \supset H_n$. Consider the $G$-space $\Delta_n \times G$ with the action $g(x,g')=(x, gg')$, where $\Delta_n$ is the standard simplex in an $n$-dimensional vector space. Define an equivalence relation $\sim $ on $\Delta_n \times G$ as follows:
\[
(x,g)\sim (x,g') \iff gH_m = g'H_m, \quad x\in \Delta_m\backslash \Delta_{m-1}, 0\leqslant m \leqslant n,
\]
where the embedding $\Delta_m\subset \Delta_n$, $0\leqslant m \leqslant n$, is given by 
$$(x_0, \ldots \,, x_m) \to (x_0, \ldots \,, x_m, 0, \ldots\,, 0).$$

The quotient
\[
\Delta_n(G; H_0,\ldots \,, H_n) = (\Delta_n\times G)|\sim 
\]
is a $G$-space with the action $g[x,g']=[x,gg']$, where $[x,g']$ is the equivalence class of $(x,g')\in \Delta_n \times G$.

The compact $G$-space $\Delta_n(G; H_0,\ldots \,, H_n)$ is called a \emph{standard $n$-dimensional $G$-simplex} of type $(H_0, H_1, \ldots \,, H_n)$. The orbit space of the $G$-simplex $\Delta_n(G; H_0,\ldots \,, H_n)$ is $\Delta_n$, with the orbit projection $\pi : \Delta_n(G; H_0,\ldots \,, H_n) \to \Delta_n$ given by $\pi([x,g]) = x$.

A $G$-space $X$ is said to be \emph{equivariantly triangulable} if there exists a triangulation $t:K\to X|G$ of the orbit space $X|G$, such that for any $n$-dimensional simplex $s$ of $K$ there exist closed subgroups $H_0\supset H_1\supset \ldots \supset H_n$ of $G$ and an equivariant homeomorphism
\[
\alpha : \Delta_n(G; H_0,\ldots \,, H_n) \to p^{-1}(t(s))
\] 
inducing a linear homeomorphism $\alpha|G : \Delta_n \to t(s)$ of the orbit space. Here $p:X\to X|G$ is the orbit projection. $G$-subsets $p^{-1}(t(s))$ are called \emph{equivariant simplices} of the equivariant triangulation of $X$. Since the simplicial complex $K$ is endowed with the weak topology, the topology of the simplicial $G$-complex $X$ will be weak with respect to the family of all equivariant simplices.

The following theorem on equivariant triangulation of $G$-spaces was proved by Illman (\cite{Illman-1}, Theorem 5.5).

\begin{theorem}[Illman \cite{Illman-1}]\label{th-Illman}
Let $X$ be a $G$-space, and let $t:K\to X|G$ be a triangulation of the quotient $X|G$ such that, for any open simplex $\mathring{s}$, the set $t(\mathring{s})\subset X|G$ has constant orbit type. Then $X$ is equivariantly triangulable with triangulation $t:K' \to X|G$, where $K'$ is the barycentric subdivision of $K$.
\end{theorem}

As in the classical case, invariant open subsets of a simplicial $G$-complex admit an equivariant simplicial structure.

\begin{theorem}\label{th-otkr}
Invariant open subsets of a simplicial $G$-complex are equivariantly triangulable.
\end{theorem}

\begin{proof}
Let $X$ be a simplicial $G$-complex with an invariant open subset  $U\subset X$. The orbit quotient $X|G$ is a simplicial complex. Since the projection  $p:X\to X|G$ is an open map, $p(U)=U|G$ is an open subset of $X|G$. Note that there exists a triangulation of $p(U)=U|G$ such that the orbit types of all open simplices in this triangulation are constant. Then Theorem \ref{th-Illman} implies that there exists an equivariant triangulation of the $G$-space $U$. The theorem is proved.
\end{proof}

Like CW-complexes, $G$-CW-complexes are constructed inductively, by attaching at each step equivariant cells of a given dimension to the result of the previous step. Thus, a $G$-space $X$ is called a \emph{$G$-CW-complex} if it is represented as a union of $G$-spaces $X^n$, $n=1,2, \ldots $\,, and the following conditions are satisfied:

1) $X^0$ is a disjoint union of orbits $G|H$, that is, a disjoint union of equivariant $0$-cells;

2) $X^{n+1}$ is obtained from $X^n$ by attaching equivariant $(n+1)$-cells $G|H \times D^{n+1}$ along equivariant maps $G|H \times S^n \to X^n$;

3) the topology of $X$ is consistent with the family $\{X^n; \ n \geqslant 0\}$.

\begin{theorem}\label{th-retract}
A $G$-CW-complex is an equivariant retract of a simplicial $G$-complex.
\end{theorem}

\begin{proof}
Let $X$ be a $G$-CW-complex. The orbit space $X|G$ is a CW-complex. By a result of Cauty (\cite{Cauty-1}, Corollary 2) there exists a closed embedding $i : X|G \to L$ in a simplicial complex $L$ such that $i(X|G)$ is a neighbourhood retract of $L$ (see \cite{Cauty-1}, Theorems 7 and 8). Since open subsets of a simplicial complex are triangulable, we can assume that $i(X|G)$ is a retract of a simplicial complex $K$. Let $r:K \to i(X|G)$ be the corresponding retraction. Consider a triangulation of $K$ such that, for any open simplex $\mathring{s}$ of it, the image $r(\mathring{s})$ lies in some cell of the CW-complex $i(X|G)$.

Now consider the subspace $r^*(X) = \{(k,x)\in K\times X; \ r(k)=p(x)\}$ of the product $K\times X$, where $p:X\to X|G$ is the natural projection. We define a $G$-action on $r^*(X)$ by $g(k,x) = (k,gx)$. This turns $r^*(X)$ into a $G$-space. Note that the orbit space $r^*(X)|G$ is the simplicial complex $K$. Let $\pi : r^*(X) \to K$ be the natural projection.

We define a map $\tilde{i} : X \to r^*(X)$ by $\tilde{i} (x) = (i(p(x)), x)$. This is an equivariant embedding of the $G$-space $X$ in $r^*(X)$. Since $i(X|G)$ is closed in $K$ and $\tilde{i} (X) = \pi^{-1}(i(X|G))$, the subset $\tilde{i} (X)$ is closed in $r^*(X)$.

Since all points in each cell of $i(X|G)$ have the same orbit type, the orbit type is constant on all open simplices of $K$. Therefore, by Theorem \ref{th-Illman} there exists an equivariant triangulation of the $G$-space $r^*(X)$. The theorem is proved.
\end{proof}

In the case of trivial $G$-action, Theorem \ref{th-retract} implies the following. 

\begin {theorem}
Any CW-complex is a retract of a simplicial complex.
\end{theorem}

This statement follows easily from the results of Cauty \cite{Cauty-1}.

Open invariant subsets of a $G$-CW-complex are not $G$-CW-complexes in general (see \cite{Cauty-2}, Examples 1 and 2). However, they are equivariant retracts of $G$-CW-complexes. This follows from a stronger result presented next.

\begin{theorem}\label{lem-0}
Any invariant open subset of a $G$-CW-complex is an equivariant retract of a simplicial $G$-complex.
\end{theorem}

\begin{proof}
Let $U$ be an invariant open subset of a $G$-CW-complex $X$. By Theorem \ref{th-retract} there exists a simplicial $G$-complex $K$ and a closed equivariant embedding $i:X\to K$ such that $i(X)$ is an equivariant retract of $K$. Let $r:K \to i(X)$ be an equivariant retraction. Since $i(U)\subset i(X)\subset K$ is an invariant open subset of $i(X)$, the subset $r^{-1}(i(U))$ is invariant and open in $K$. Furthermore, $i(U)$ is an equivariant retract of $r^{-1}(i(U))$, which is a simplicial $G$-complex by Theorem \ref{th-otkr}. The theorem is proved.
\end{proof}

\section{Equivariant fibrations}

Let $E$ and $B$ be $G$-spaces. An equivariant map $p : E \to B$ is said to have the \emph{equivariant covering homotopy property} with respect to a $G$-space $X$ if for arbitrary
equivariant maps $\widetilde{f}:X \to E$ and $F: X\times I \to B$ such that $F(x,0)=p\widetilde{f}(x)$, $x\in X$,
there exists an equivariant map $\widetilde{F} : X\times I \to  B$, satisfying $\widetilde{F}(x,0)=\widetilde{f}(x)$ for any $x\in X$ and $p\circ \widetilde{F} =F$.

An equivariant map $p : E \to B$ is called an \emph{equivariant Hurewicz fibration} if $p$ has the equivariant covering homotopy property with respect to any $G$-space $X$. If an equivariant map $p : E \to B$ has the equivariant covering homotopy property with respect to $G$-CW-complexes, then it is called an \emph{equivariant Serre fibration}.

\begin{lemma}\label{prop-1}
Let $p:E\to B$ be an equivariant Serre fibration, and let $X$ be an equivariant retract of a simplicial $G$-complex $K$. Then $p:E\to B$ has the equivariant covering homotopy property with respect to the $G$-space $X$.
\end{lemma}

\begin{proof}
Let $f:X\to E$ be an equivariant map, and let $F:X\times I \to B$ be an equivariant homotopy of the map $p \circ f$: $F_0=p \circ f$. Consider an equivariant retraction $r:K \to X$ and an equivariant homotopy $H:K\times I \to B$ given by $H=F\circ (r\times \rm{id})$. Clearly, $H_0=p\circ (f\circ r)$. Since $p:E\to B$ is an equivariant Serre fibration, there exists a covering homotopy $\widetilde{H}:K\times I \to E$ such that $\widetilde{H}_0=f\circ r$ and $p\circ \widetilde{H}=H$. Now it is easy to see that $\widetilde{F}=\widetilde{H}|_{X\times I}:X\times I \to E$ is the required
equivariant covering homotopy, that is, $\widetilde{F}_0=f$ and $p\circ \widetilde{F}= F$. The lemma is proved.
\end{proof}

\begin{definition}\label{def-EULC}
A $G$-space $X$ is said to be \emph{equivariantly uniformly locally contractible} if there exists an invariant neighbourhood $V$ of the diagonal $\Delta \subset X \times X$ and an equivariant map $\lambda : V \times I \to X$ such that

(i) $\lambda (x, y, 0) = x$ and $\lambda (x, y, 1) = y$ for all $(x, y) \in V$; 

(ii) $\lambda (x, x, t) = x$ for all $x\in X$ and $t\in I$.
\end{definition}

It is easy to see that the class of equivariantly uniformly locally contractible spaces contains all $G$-ANR-spaces. Therefore, $G$-CW-complexes are equivariantly uniformly locally contractible.

The previous results allow us to prove the following theorem.

\begin{theorem}\label{th-rassloenie}
Let $p:E\to B$ be an equivariant Serre fibration, where $E$ and $B$ are $G$-CW-complexes. Then $p:E\to B$ is an equivariant Hurewicz fibration in the class of all compactly generated spaces.
\end{theorem}

\begin{proof}
It is known that when B is a paracompact $G$-space, an equivariant locally trivial bundle  $p:E\to B$ is an equivariant Hurewicz fibration (see \cite{tom Dieck_1}, Example 5). Therefore, it is enough to prove that for an arbitrary point $b\in B$ there exists an invariant open neighbourhood $U$ of $b$ and an equivariant map   $\omega : U \times p^{-1}(U) \to E$ such that

(1) $p\circ \omega (b, e) = b$ for all $(b, e) \in U \times p^{-1} (U)$, 

(2) $\omega (p (e), e) = e$ for all $e \in p^{-1} (U)$.

Since the $G$-CW-complex B is equivariantly uniformly locally contractible (see Definition \ref{def-EULC}), there exists an invariant open neighbourhood $U$ of $b\in B$ and an equivariant map 
$$\lambda: U \times U \times I \to B,$$
such that
$$\lambda (x, y, 0) = x, \quad \lambda (x, y, 1) = y, \quad \lambda (x, x, t) = x$$ 
for all $x,y \in U$ and $t\in I$.

Now consider the equivariant map
$$F: U \times  p^{-1} (U) \times I \to B,$$ 
given by
$$F (b, e, t ) = \lambda (p (e), b, t)$$ 
for all $(b, e, t )\in U \times  p^{-1} (U) \times I$. Since $F(b,e,0)=\lambda (p (e), b, 0) = p(e) = p( pr_2(b,e,0))$, the map $F$ is an equivariant homotopy of the map  $p\circ pr_2$.

Since $U$ is an equivariant retract of a simplicial $G$-complex (see Theorem \ref{lem-0}) and $p:E\to B$ is an equivariant Serre fibration, Lemma \ref{prop-1} implies that there is an equivariant covering homotopy
$$\tilde{F}: U \times p^{-1} (U) \times I \to E,$$
satisfying $\tilde{F}(b,e,0)=e$ and $p\circ \tilde{F} = F$. Now we define the required equivariant
map $\omega : U \times p^{-1}(U) \to E$ by the formula
$$\omega (b, e) = \tilde{F}(b,e,1).$$
Conditions (1) and (2) are verified easily. The theorem is proved.
\end{proof}

\section{H-fixed point sets and G-fibrations}

Let $p:E\to B$ be a Hurewicz $G$-fibration. Given a closed subgroup $H$ of $G$, the induced $H$-fixed point map $p^H : E^H \to B^H$ is a Hurewicz fibration in the usual sense. A natural question arises: is this necessary condition also sufficient? This question was posed by James and Segal in \cite{james-segal}.

Why this question is natural is explained, in particular, by the following well-known result (for example, see \cite{luck}, § 15.3).

\begin{theorem}\label{th-serr}
An equivariant map $p:E\to B$ is a Serre $G$-fibration if and only if the map $p^H : E^H \to B^H$ is a Serre fibration for any closed subgroup $H$ of $G$.
\end{theorem}

In the case of a finite group $G$ Theorem \ref{th-serr} was proved by Bredon in \cite{Bredon}, Ch. III, 4.1.

\begin{theorem}\label{th-nepodvij-serr}
Let $p:E\to B$ be an equivariant map of $G$-CW-complexes. If the map $p^H : E^H \to B^H$ is a Serre fibration for any closed subgroup $H$ of $G$, then $p:E\to B$ is a Hurewicz $G$-fibration in the class of compactly generated $G$-spaces.
\end{theorem}

\begin{proof}
The fact that  $p^H : E^H \to B^H$ is a Serre fibration for any closed subgroup $H$ implies that the equivariant map $p:E\to B$ is a Serre $G$-fibration (see Theorem \ref{th-serr}). Since both $E$ and $B$ are $G$-CW-complexes, it follows from Theorem \ref{th-rassloenie} that the equivariant map $p:E\to B$ is a Hurewicz $G$-fibration in the class of compactly generated $G$-spaces. The theorem is proved.
\end{proof}

From the latter theorem we obtain the following positive answer to the problem of James and Segal in the class of compactly generated $G$-spaces and under the condition that both $E$ and $B$ are $G$-CW-complexes:

\begin{theorem}\label{th-nepodvij}
An equivariant map $p:E\to B$ of $G$-CW-complexes is a Hurewicz $G$-fibration in the class of compactly generated $G$-spaces if and only if the map $p^H : E^H \to B^H$ is a Hurewicz fibration for any closed subgroup $H$ of $G$.
\end{theorem}

The question remains open as to whether Theorem \ref{th-nepodvij} holds in the case when $E$ and $B$ are arbitrary $G$-spaces.


\begin{thebibliography}{99}

\bibitem{St-West}
M.~Steinberger, J.~West,
\emph{Covering homotopy properties of maps between CW complexes or ANR's}, Proceedings of the American Mathematical Society, 92:4, 573--577 (1984).

\bibitem{Cauty-2}
R.~Cauty, 
\emph{Sur les ouverts des CW-complexes et les fibr\'es de Serre}, Colloquium Mathematicum, 63, 1--7 (1992).

\bibitem{May-Sig} 
J.~May, J.~Sigurdsson,
\emph{Parametrized homotopy theory}, 132, American Mathematical Soc., 2006.

\bibitem{james-segal}
I.\,M.~James, G.\,B.~Segal, 
\emph{On Equivariant homotopy type}, Topology, 17, 267--272 (1978).

\bibitem{Bredon} 
G.~Bredon, 
\emph{Equivariant cohomology theories}, Springer Lecture Notes in Mathematics, 34 (1967).

\bibitem{Illman-1}
S.~Illman,
\emph{The Equivariant Triangulation Theorem for Actions of Compact Lie Groups}, Math. Ann., 262:4, 487--501 (1983).

\bibitem{Cauty-1}
R.~Cauty,  
\emph{Sur les sous-espaces des complexes simpliciaux}, Bull. Soc. Math. France, 100,  129--155, (1972).

\bibitem{tom Dieck_1} 
T.~tom Dieck,
\emph{Transformation groups}, Walter de Gruyter, 8 (1987).

\bibitem{luck}
W.~Luck,
\emph{Transformation Groups and Algebraic $K$-Theory}, Lecture Notes in Mathematics, 1408 (1989).
\end{thebibliography}
\end{document}